\documentclass[12pt,oneside,reqno]{article}

\usepackage{braket, amsxtra, xr}
\usepackage[utf8]{inputenc}

\usepackage{slashed}
\usepackage{graphicx}
\usepackage{xcolor}
\usepackage{pdfrender}
\pdfrender{TextRenderingMode=2,LineWidth=.5pt}
\usepackage{amsmath}
\usepackage{amsthm}
\usepackage{indentfirst}
\usepackage{amsfonts}
\usepackage{marginnote}
\usepackage{xcolor}
\usepackage{hyperref}
\usepackage{OldStandard}
\usepackage{multicol}
\usepackage{lipsum}
\usepackage{newtxmath}

\usepackage{geometry}
\newcommand\mymathop[1]{\mathop{\operatorname{#1}}}

\usepackage{blindtext}
\usepackage{authblk}
\usepackage{mathtools}

\usepackage{xcolor}
\usepackage{pdfrender}
\pdfrender{TextRenderingMode=2,LineWidth=.5pt}

\usepackage{lipsum} % for mock text

\newtheoremstyle{plainindent}
  {\topsep}   % ABOVESPACE
  {\topsep}   % BELOWSPACE
  {\itshape}  % BODYFONT
  {\parindent}% INDENT (empty value is the same as 0pt)
  {\bfseries} % HEADFONT
  {.}         % HEADPUNCT
  {5pt plus 1pt minus 1pt} % HEADSPACE
  {}          % CUSTOM-HEAD-SPEC

\theoremstyle{plainindent}

\parskip=\baselineskip

\newtheorem{theorem}{Theorem}
\newtheorem{corollary}{Corollary}[theorem]
\newtheorem{conjecture}{Conjecture}
\newtheorem{proposition}{Proposition}

\newcommand\myshade{85}
\colorlet{mylinkcolor}{blue}
\colorlet{mycitecolor}{red}
\hypersetup{linkcolor  = mylinkcolor,
	citecolor  = mycitecolor,
	urlcolor   = myurlcolor!\myshade!black,
	colorlinks = true,}

\title{\textsc{ a generalized ramanujan master theorem and integral representation of meromorphic functions}}

\author{\textsc{Zachary P. Bradshaw\thanks{E-mail: zbradshaw@tulane.edu}
 \,and Omprakash Atale\thanks{E-mail: atale.om@outlook.com}}}
\affil{\small ${}^{*}$\textit{Naval Surface Warfare Center, Panama City Division} \\ \textit{Panama City, FL, United States } \\${}^{\dagger}$\textit{Department of Mathematics, Savitribai Phule Pune University,}\\ \textit{Pune-425001, India}}

\begin{document}
\maketitle
\begin{abstract}
   Ramanujan's Master Theorem is a decades-old theorem in the theory of Mellin transforms which has wide applications in both mathematics and high energy physics. The unconventional method of Ramanujan in his proof of the theorem left convergence issues which were later settled by Hardy. Here we extend Ramanujan's theorem to meromorphic functions with poles of arbitrary order and observe that the new theorem produces analogues of Ramanujan's famous theorem. Moreover, we find that the theorem produces integral representations for meromorphic functions which are shown to satisfy interesting properties, opening up an avenue for further study.
\end{abstract}
\newpage
\fontsize{11.5}{16}\selectfont
\begin{center}
    \textsc{\textbf{\large \S I. Introduction}}
\end{center}
It is no secret that the famous Indian mathematician Srinivasa Ramanujan was known for his mathematical insight. However, without formal training in the subject, his work was largely non-rigorous, leaving much for future generations of mathematicians to complete. Here we will expand upon just one piece of Ramanujan's work, which he readily used to explore definite integrals. It is called Ramanujan's master theorem, the name of which can be attributed to Berndt \cite{3}, and it states that if $f(x)$ has an expansion of the form
\begin{equation*}
f(x)=\sum_{n=0}^{\infty}(-1)^{n} \frac{g(n)}{n !} x^{n}\tag{1.1}
\end{equation*}
where $g(n)$ has a natural and continuous extension such that $g(0) \neq 0$, then for $s>0$, we have
\begin{equation}\label{eq:rmtintro}
\int_{0}^{\infty} x^{s-1}\left(\sum_{n=0}^{\infty}(-1)^{n} \frac{g(n)}{n !} x^{n}\right) dx =g(-s)\Gamma(s).\tag{1.2}
\end{equation}

Eqn. \eqref{eq:rmtintro} can be applied to calculate the Mellin transform of many functions in the literature. Ramanujan first communicated this result to Hardy in his 1913 quarterly reports, but his method for proving it was incomplete. It started with Euler's integral representation of the gamma function and relied on Taylor's expansion without regard for convergence issues. This problem was remedied by Hardy when he made natural assumptions on the analyticity of $g(n)$ as well as its growth rate \cite{21} which facilitated a proof of the theorem using Cauchy's residue theory and the Mellin inversion theorem (see Theorem~\ref{thm:RMT}).

At present, Hardy's class of functions has been extended in many ways \cite{6,8}, and several generalizations of Ramanujan's theorem exist \cite{5,10}, including a multidimensional version known as the method of brackets \cite{bradshaw2023,bradshaw2023a,15,16,17,18,19}, which was discussed in \cite{23} and more recently in \cite{1}. This method has found use in the field of high energy particle physics and quantum field theory \cite{20,25} for its application to precision calculations involving Feynman diagrams. Indeed, the method of brackets is particularly well equipped to handle the Schwinger parametrization \cite{25} of loop integrals.

The purpose of this article is to point out that Hardy's proof can be generalized in a natural way to produce a collection of Ramanujan-like master theorems. We will see that the crux of Hardy's proof is a residue calculation followed by the Mellin inversion theorem \cite{7}, the former of which we will alter. In doing so, we find that the proof still holds under a modification of the growth conditions outlined by Hardy. As a result, we obtain many formulas for the Mellin transform of a function in terms of an analytic continuation of its coefficients, and these results therefore have applications to sequence interpolation.

The Mellin transform, being scale-invariant, is widely used in the analysis of computer algorithms \cite{14}. Moreover, in quantum theory, the Fourier transform is used to transform from position space to momentum space, where computations may be easier, and the Mellin transform plays an analogous role in the AdS/CFT correspondence \cite{11,12}. Additional applications include the asymptotic approximations of functions which are defined by integrals \cite{4,26} and series \cite{13,22}, the study of distributions of products and quotients of independent random variables \cite{9}, and the expression of solutions to electromagnetic wave propagation in turbulence \cite{24}. The results presented here are therefore of interest to a wide range of disciplines.

The remainder of this work is outlined as follows. In \S II, we review Ramanujan's master theorem and the proof given by Hardy, after which we discuss the interpretation of the theorem as a sequence interpolation formula. We then modify Hardy's proof in \S III, producing a family of Ramanujan-like master theorems, which we may interpret as interpolation formulas. In \S IV, we study independently some properties of the integral representation of $h(z)$. In particular, we show that it is logarithmically convex. We extend further in \S V to meromorphic functions with poles of arbitrary order, producing several exotic integral representations for meromorphic functions. Finally, in \S VI, we give concluding remarks.

\begin{center}
    {\large\textsc{\textbf{\S II. Ramanujan's Master Theorem}}}
\end{center}
We begin with Ramanujan's master theorem and its proof given by Hardy \cite{21}. Ultimately, the theorem gives an expression for the Mellin transform of a certain class of functions; though, it can also be viewed as an interpolation formula. Indeed, given a sequence $c_n$, we define a function $g:\mathbb{N}\to\mathbb{C}$ by $g(n)=c_n$ and ask how this function might be extended to a larger subset of the complex plane. It turns out that Ramanujan's master theorem gives an extension of this function in a unique way, and we will return to this point in a moment. For now, we turn to the theorem.

\begin{theorem}[\textbf{Hardy-Ramanujan}]\label{thm:RMT} Let $g(z)$ be analytic on the half-plane $H(\delta)=\{z\in\mathbb{C}|\textnormal{Re}(z)\ge-\delta\}$ for some $0<\delta<1$. Assume $A<\pi$ and $g$ satisfies the growth condition
\begin{align} |g(v+iw)|<Ce^{Pv+A|w|} \tag{2.1} \label{eq:growth}
\end{align}
for all $z=v+iw\in H(\delta)$. Then
\begin{align} \int_0^\infty x^{s-1}\sum_{k=0}^\infty(-1)^kg(k)x^kdx=\frac{\pi}{\sin(\pi s)}g(-s) \tag{2.2}\label{eq:rmt}
\end{align}
holds for all $0<\textnormal{Re}(s)<\delta$.
\end{theorem}
\begin{proof} If $|x|<e^{-P}$, then the series converges by the root test. Let $0<c<\delta$ and consider the contour integral
\begin{equation*}
    \frac{1}{2\pi i}\int_C\frac{\pi}{\sin(\pi z)}g(-z)x^{-z}dz \tag{2.3}
\end{equation*}
where $C$ is the left semicircle contour of radius $n\in\mathbb{N}$ centered at $z=c$. The integral $I_\gamma$ along the arc $\gamma$ of the semicircle satisfies
\begin{align*} \lvert I_\gamma\rvert&=\bigg|\int_\gamma\frac{\pi}{\sin(\pi z)}g(-z)x^{-z}\ dz\bigg|\\
&=\bigg|\int_{\pi/2}^{3\pi/2}\frac{\pi g(-c-ne^{i\theta})x^{-c-ne^{i\theta}}}{\sin(\pi(c+ne^{i\theta}))}nie^{i\theta}d\theta\bigg|\\
&\le\int_{\pi/2}^{3\pi/2}\frac{C\pi ne^{An|\sin(\theta)|}d\theta}{\sqrt{\sin^2(\pi c+\pi n\cos(\theta))+\sinh^2(\pi n\sin(\theta))}}. \tag{2.4}
\end{align*} 
Now taking the limit, as $n\to\infty$, we see that the arc term has zero contribution to the integral.
The integrand has simple poles at $-n$ for all $n\in\mathbb{N}$ and the residues are 
\begin{align*}\lim_{z\to-n}\frac{\pi(z+n)g(-z)x^{-z}}{\sin(\pi z)}&=\pi\lim_{z\to-n}g(-z)\lim_{z\to-n}\frac{(z+n)}{x^{z}\sin(\pi z)}\\
&=\pi g(n)\lim_{z\to-n}\frac{1}{\ln(x)x^z\sin(\pi z)+\pi x^z\cos(\pi z)}\\
&=\pi g(n)\frac{1}{\pi(-1)^nx^{-n}}\\
&=(-1)^ng(n)x^n. \tag{2.5}
\end{align*}
Then by the residue theorem, we have
\begin{align}\frac{1}{2\pi i}&\int_{c-i\infty}^{c+i\infty}\frac{\pi g(-s)x^{-s}}{\sin(\pi s)}ds=\sum_{k=0}^\infty(-1)^kg(k)x^k. \label{eq:residueintegral} \tag{2.6}
\end{align}
The function $\frac{\pi}{\sin(\pi z)}g(-z)$ is analytic on $0<\textnormal{Re}(z)<\delta$ and the integral converges absolutely and uniformly for $c\in(a,b)$ and $0<a<b<\delta$. Then by the Mellin inversion theorem,
\begin{align*} \frac{\pi}{\sin(\pi s)}g(-s)=\int_0^\infty x^{s-1}\sum_{k=0}^\infty(-1)^kg(k)x^kdx \tag{2.7}
\end{align*}
for all $0<\textnormal{Re}(s)<\delta$.
\end{proof}

It is now clear how Ramanujan's master theorem can be viewed as an interpolation formula. Given $g(n)=c_n$ as before, we define $g(z)$ in the region $-\delta<\textnormal{Re}(z)<0$ by \eqref{eq:rmt}. We can do even better by analytically continuing the domain of validity of the theorem to the interval between any two consecutive negative integers (see Theorem 8.1 in \cite{1}) by shifting the contour in the proof of the master theorem. In shifting the vertical leg of the contour to the left, we drop some of the singularities of the reciprocal sine function from the contour, resulting in the interpolation formula
\begin{align}
\int_0^\infty x^{s-1}\sum_{k=N}^\infty(-1)^kg(k)x^k\ dx=\frac{\pi}{\sin(\pi s)}g(-s) \label{eq:extendedrmt} \tag{2.8}
\end{align}
in the region $-N<\textnormal{Re}(s)<-N+1$.

Note that both the assumption of analyticity on $H(\delta)$ and the growth condition \eqref{eq:growth} are unnecessary to achieve \eqref{eq:extendedrmt}. We may weaken these assumptions to hold only in the region we integrate, which shrinks with increasing $N$. Similarly, if $g(n)$ is defined for $n\in\mathbb{Z}_{<0}$, by shifting the vertical leg of the contour to the right and strengthening Hardy's assumptions in the opposite way, we recover the interpolation formula
\begin{align*}
\int_0^\infty x^{s-1}\sum_{k=-N}^\infty(-1)^kg(k)x^k\ dx=\frac{\pi}{\sin(\pi s)}g(-s) \tag{2.9}
\end{align*}
in the region $N<\textnormal{Re}(s)<N+1$.

Observe that the main idea of Hardy's proof is to apply the residue theorem to compute the integral \eqref{eq:residueintegral} and then apply the Mellin inversion theorem. It happens that for the choice of $\pi/\sin(\pi s)$ in the integrand, the residue computation is such that the result of \eqref{eq:residueintegral} is a power series in $x$, giving Ramanujan's master theorem the particularly nice form \eqref{eq:rmt}. However, nothing is stopping us from replacing this function with another singular function at the non-positive integers.

With this in mind, let us replace $\pi/\sin(\pi s)$ by a meromorphic function $h(s)$ with poles at the non-positive integers. The result of the residue calculation cannot be computed explicitly without additional information about $h$, but it is clear from Hardy's proof that after applying the Mellin inversion theorem, the power series in the integrand will now have the form
\begin{align*}
\sum_{k=0}^\infty\mymathop{Res}_{z=-k}\left(\frac{h(z)g(-x)}{x^z}\right) \tag{2.10}
\end{align*}
so that the result of the theorem will be of the form
\begin{align}\label{eq:formula-type}
\int_0^\infty x^{s-1}&\sum_{k=0}^\infty\mymathop{Res}_{z=-k}\left(\frac{h(z)g(-x)}{x^z}\right)dx=h(s)g(-s). \tag{2.11}
\end{align}
However, we have glossed over the fine details. For one, the region of convergence of the series will in general change with $h$. More importantly, the growth condition \eqref{eq:growth} is no longer enough to guarantee that the contribution from the arc term $I_\gamma$ of the contour integral falls off at infinity. Therefore, we will have to be careful to tweak the conditions on $g$ according to the chosen $h$ when constructing analogs of Ramanujan's theorem.

\begin{center}
    \textsc{\textbf{\S III. Simple poles}}
\end{center}
Under the assumption that the poles of $h(z)$ are simple, the formula \eqref{eq:formula-type} simplifies quite nicely, and it is not too hard to adjust Hardy's proof to this more general setting. Indeed, we have the following theorem for meromorphic functions with simple poles at the non-positive integers. Let $\phi(z)$ denote the analytic continuation of $\sin(\pi z)h(-z)$ to the non-negative integers; there are removable singularities here which are otherwise problematic when trying to evaluate this function.

\begin{theorem}[\textbf{Main Theorem}]\label{thm:meromorphic-formula}
Let $g(z)$ be analytic on the half-plane $H(\delta)=\{z\in\mathbb{C}|\textnormal{Re}(z)\ge-\delta\}$ for some $0<\delta<1$, and let $h(z)$ be meromorphic with simple poles at the non-positive integers. Assume $A<\pi$ and $g$ satisfies the growth condition
\begin{align} \lvert g(v+iw)\rvert<\frac{Ce^{Pv+A|w|}}{\lvert\phi(v+iw)\rvert} \label{eq:growthsimple} \tag{3.1}
\end{align}
for all $z=v+iw\in H(\delta)$ and that
\begin{align}\label{eq:limit-property}
    \lim_{k\to\infty}\left|\frac{\textnormal{Res}_{-k}(h(z))}{\phi(k)}\right|^{1/k}=L\ge1.\tag{3.2}
\end{align} Then
\begin{align*}\int_0^\infty x^{s-1}&\sum_{k=0}^\infty\mymathop{Res}_{z=-k}\left(h(z)\right)g(k)x^kdx=h(s)g(-s).\tag{3.3}
\end{align*}
holds for all $0<\textnormal{Re}(s)<\delta$.
\end{theorem}
\begin{proof}  Observe that 
\begin{align*}
    \left|\mymathop{Res}_{z=-k}(h(z))g(k)x^k\right|^{1/k}
    =\left|\mymathop{Res}_{z=-k}(h(z))\right|^{1/k}|g(k)|^{1/k}|x|\le\left|\dfrac{\mymathop{Res}\limits_{z=-k}(h(z))}{\phi(k)}\right|^{1/k}C^{1/k}e^P|x|,\tag{3.4}
\end{align*}
and taking the limit as $k\to\infty$ yields $Le^P|x|$. Thus, if $|x|<\frac{e^{-P}}{L}$, then the series converges by the root test. Let $0<c<\delta$ and consider the contour integral
\begin{equation*}
    \frac{1}{2\pi i}\int_Ch(z)g(-z)x^{-z}dz \tag{3.5}
\end{equation*}
 where $C$ is the left semicircle contour of radius $n\in\mathbb{N}$ centered at $z=c$.
Now consider the integral along the arc $\gamma$ of the semicircle. We have
\begin{align*} \bigg|\int_\gamma h(z)g(-z)x^{-z}\ dz\bigg|
&=\bigg|\int_{\pi/2}^{3\pi/2}h(c+ne^{i\theta})g(-c-ne^{i\theta})x^{-c-ne^{i\theta}}\ nie^{i\theta}d\theta\bigg|\\
&< \int_{\pi/2}^{3\pi/2}\frac{C ne^{An|\sin(\theta)|}L^{c+n\cos(\theta)}d\theta}{\sqrt{\sin^2(k(\theta))+\sinh^2(\pi n\sin(\theta))}}, \tag{3.6}
\end{align*}
where $k(\theta)=\pi c+\pi n\cos(\theta)$, and we have used the growth condition \eqref{eq:growthsimple} and the earlier derived bound $|x|<e^{-P}/L$. Using the fact that $\cos(\theta)$ is negative on the interval $(\pi/2,3\pi/2)$ and $L\ge1$ by assumption, and taking the limit as $n\to\infty$, we see that the arc term has zero contribution to the integral. Then by the residue theorem, we have
\begin{align*}\frac{1}{2\pi i}&\int_{c-i\infty}^{c+i\infty}x^{-s}g(-s)h(s)ds=\sum_{k=0}^\infty\mymathop{Res}_{z=-k}(h(z))g(k)x^{k}. \tag{3.7}
\end{align*}
Now, $x^{-s}g(-s)h(s)$ is analytic on $0<\mathrm{Re}(s)<\delta$ and the integral converges absolutely and uniformly for $c\in(a,b)$ and $0<a<b<\delta$. Then by the Mellin inversion theorem,
\begin{align*} \int_0^\infty x^{s-1}\sum_{k=0}^\infty\mymathop{Res}_{z=-k}(h(z))g(k)x^{k}dx=h(s)g(-s) \tag{3.8}
\end{align*}
for all $0<\mathrm{Re}(s)<\delta$.
\end{proof}

Note that setting $g=1$ produces an integral representation for meromorphic functions with simple poles so long as they satisfy the remaining hypotheses of the theorem:
\begin{align}\label{eq:meromorphic-rep}
    h(s)=\int_0^\infty x^{s-1}\sum_{k=0}^\infty\mymathop{Res}_{z=-k}(h(z))x^k\ dx. \tag{3.9}
\end{align}
In particular, we recover Bernoulli's integral representation of the Gamma function by setting $h(s)=\Gamma(s)$. Indeed, $\mymathop{Res}\limits_{z=-k}(\Gamma(z))=(-1)^k/k!$, so that
\begin{align}
    \Gamma(s)&=\int_0^\infty x^{s-1}\sum_{k=0}^\infty\frac{(-1)^k}{k!}x^k\ dx=\int_0^\infty x^{s-1}e^{-x}\ dx. \label{eq:bernoulli}\tag{3.10}
\end{align}

Interestingly, if Theorem~\ref{thm:meromorphic-formula} were only known to be true for the case $g=1$, the case of more general $g$ can also be derived using a heuristic method involving Taylor's theorem, which Ramanujan used in his quarterly reports. Starting with $h(s)$ as defined by \eqref{eq:meromorphic-rep}, replace $x$ with $y^nx$ so that we have
\begin{align*}
    y^{-ns}h(s)=\int_0^\infty x^{s-1}\sum_{k=0}^\infty\mymathop{Res}_{z=-k}(h(z))y^{nk}x^k\ dx. \tag{3.11}
\end{align*}
Defining $f$ by $g(z)=f(y^z)$, multiplying both sides by $\frac{f^{(n)}(0)}{n !}$, and summing over $0 \leq n<\infty$ produces
\begin{align*}
    \sum_{n=0}^\infty\frac{f^{(n)}(0) (y^{-s})^{n}}{n !}h(s)&=\sum_{n=0}^\infty\frac{f^{(n)}(0)}{n!}\int_0^\infty x^{s-1}\sum_{k=0}^\infty\mymathop{Res}_{z=-k}(h(z))y^{kn}x^k\ dx\\
    &=\int_0^\infty x^{s-1}\sum_{k=0}^\infty\mymathop{Res}_{z=-k}(h(z))\sum_{n=0}^\infty\frac{f^{(n)}(0) (y^k)^{n}}{n !}x^k\ dx \tag{3.12}
\end{align*}
From the definition of $f$, it follows that
\begin{align*} 
g(-s)h(s)=\int_0^\infty x^{s-1}\sum_{k=0}^\infty \mymathop{Res}_{z=-k}(h(z))g(k)x^kdx, \tag{3.13}
\end{align*}
thereby re-establishing the more general Theorem~\ref{thm:meromorphic-formula}. 

Just as in the case of Ramanujan's master theorem, Theorem~\ref{thm:meromorphic-formula} can be considered a sequence interpolation formula, and using Carlson's theorem \cite{bradshaw2023,hardy1920} we can show that this extension is unique.
\begin{proposition}\label{thm:unique-hardy-type} Let $g(n)$ be a function defined on the non-negative integers. Let $f_1(z)$ and $f_2(z)$ be analytic extensions of $g(n)$ satisfying \eqref{eq:growthsimple}. Then $f_1=f_2$.
\end{proposition}
\begin{proof} Define $k(z):=f_1(z)-f_2(z)$. Then for all $n\in\mathbb{N}$, we have
$k(n)=f_1(n)-f_2(n)=g(n)-g(n)=0$. Furthermore, since $f_1$ and $f_2$ satisfy \eqref{eq:growthsimple}, $k$ also satisfies \eqref{eq:growthsimple}. Then there exists $C,P,A\in\mathbb{R}$ with $A<\pi$ such that 
\begin{align*}\lvert k(z)\rvert\cdot\lvert\phi(z)\rvert&\le Ce^{Px+A|y|}\\
&\le Ce^{P|z|+A|z|}\nonumber\\
&\le Ce^{2\max\{P,A\}|z|}\nonumber
\end{align*}
Thus, by Carlson's theorem, $k$ is identically zero. i.e. $f_1=f_2$.
\end{proof}

We now list several corollaries of Theorem~\ref{thm:meromorphic-formula}, some of which already appear in the literature. To start, let $h(z)=\psi(z)$ so that $\phi(z)=\sin(\pi z)\psi(-z)$.

\begin{corollary}\label{cor:digamma}
Let $g(z)$ be analytic on the half-plane $H(\delta)=\{z\in\mathbb{C}|\textnormal{Re}(z)\ge-\delta\}$ for some $0<\delta<1$. Assume $A<\pi$ and $g$ satisfies the growth condition
\begin{align*} \lvert g(v+iw)\rvert<\frac{Ce^{Pv+A|w|}}{\lvert\phi(v+iw)\rvert} \tag{3.14}
\end{align*}
for all $z=v+iw\in H(\delta)$. Then
\begin{align*} \int_0^\infty x^{s-1}\sum_{k=0}^\infty g(k)x^kdx=-\psi(s)g(-s) \tag{3.15}
\end{align*}
holds for all $0<\textnormal{Re}(s)<\delta$.
\end{corollary}
\begin{proof}
     Set $h(z)=\psi(z)$ in Theorem ~\ref{thm:meromorphic-formula}, where $\psi$ denotes the digamma function defined as the logarithmic derivative of the gamma function
\begin{align*}
    \psi(z)&:=\frac{d}{dz}\ln(\Gamma(z))=\frac{\Gamma'(z)}{\Gamma(z)}. \tag{3.16}
\end{align*}
Note that the digamma function has simple poles at the non-positive integers. A simple residue computation reveals that
\begin{align*}
    \mymathop{Res}_{z=-k}(\psi(z)g(-z)x^{-z})=-g(k)x^k, \tag{3.17}
\end{align*}
so that the resulting power series appearing in the integrand of the theorem will be $-\sum_{k=0}^\infty g(k)x^k$. Moreover, a simple computation shows that this choice of function satisfies property \eqref{eq:limit-property} with $L=1$.
\end{proof}

Corollary~\ref{cor:digamma} appears to be more of an interesting artifact of Theorem~\ref{thm:meromorphic-formula} than a useful tool as there don't appear to be many functions which satisfy its hypothesis, if any. Not even the choice $g=1$ suffices here. Let $h(z)=\Gamma(s)\cos(\pi s/2)$. The following is a more useful identity which appears in \cite{2}, and the reader can find an operational justification for it in \cite{5}. 

\begin{corollary}\label{cor:cosine}
Let $g(z)$ be analytic on the half-plane $H(\delta)=\{z\in\mathbb{C}|\textnormal{Re}(z)\ge-\delta\}$ for some $0<\delta<1$. Assume $A<\pi$ and $g$ satisfies the growth condition
\begin{align*} \lvert g(v+iw)\rvert<\frac{Ce^{Pv+A|w|}}{\lvert\phi(v+iw)\rvert}  \tag{3.18}
\end{align*}
for all $z=v+iw\in H(\delta)$. Then
\begin{align*} \int_0^\infty x^{s-1}&\sum_{k=0}^\infty\frac{(-1)^k}{\Gamma(2k+1)}g(2k)x^{2k}dx=\Gamma(s)g(-s)\cos\left(\frac{\pi s}{2}\right) \tag{3.19} 
\end{align*}
holds for all $0<\textnormal{Re}(s)<\delta$.
\end{corollary}
\begin{proof}
    The function $x^{-s}g(-s)\cos(\frac{\pi s}{2})\Gamma(s)$ has simple poles at the non-positive even integers (the odd integer poles are canceled by the zero of the cosine function) and the residues are given by
\begin{align*}
    \mymathop{Res}_{z=-2k}&\left(x^{-s}g(-s)\cos\left(\frac{\pi s}{2}\right)\Gamma(s)\right)=(-1)^k\frac{g(2k)}{\Gamma(2k+1)}x^{2k}. \tag{3.20}
\end{align*}
A straightforward computation shows that the condition \eqref{eq:limit-property} on $L$ is satisfied in this case with $L=1$.
\end{proof}

Notice that setting $g(z)=a^z$ gives us the Mellin transform of the scaled cosine function
\begin{align*}
    \int_0^\infty x^{s-1}\cos(ax)\ dx=a^{-s}\Gamma(s)\cos\left(\frac{\pi s}{2}\right). \tag{3.21}
\end{align*}

The bound on $g$ in Theorem~\ref{thm:meromorphic-formula} can likely be improved upon. To see this, note that an equivalent version of Ramanujan's master theorem can be derived by applying Euler's reflection formula to $\pi/\sin(\pi s)$ and redefining $g$. Indeed, by sending $g(s)\to g(s)/\Gamma(1+s)$, we have
\begin{align}\label{eq:rmt-gamma}
    \int_0^\infty x^{s-1}\sum_{k=0}^\infty\frac{(-1)^k}{k!}g(k)x^k dx=\Gamma(s)g(-s). \tag{3.22}
\end{align}
However, we can derive this result directly using a modification of Hardy's proof as before by setting $h(z)=\Gamma(z)$.

\begin{corollary}
Let $g(z)$ be analytic on the half-plane $H(\delta)=\{z\in\mathbb{C}|\textnormal{Re}(z)\ge-\delta\}$ for some $0<\delta<1$. Assume $A<\pi$ and $g$ satisfies the growth condition
\begin{align} \lvert g(v+iw)\rvert<\frac{Ce^{Pv+A|w|}}{\lvert\phi(v+iw)\rvert} \label{eq:growthgamma}\tag{3.23}
\end{align}
for all $z=v+iw\in H(\delta)$. Then
\begin{align*} \int_0^\infty x^{s-1}\sum_{k=0}^\infty\frac{(-1)^k}{k!}g(k)x^kdx=\Gamma(s)g(-s) \tag{3.24}
\end{align*}
holds for all $0<\textnormal{Re}(s)<\delta$.
\end{corollary}
\begin{proof}
    Let $h(z)=\Gamma(z)$ and apply Theorem~\ref{thm:meromorphic-formula}. The residue calculation produces
\begin{align*}
    \mymathop{Res}_{z=-k}(\Gamma(z)g(-z)x^{-z})&=g(k)x^k\mymathop{Res}_{z=-k}(\Gamma(z))\\
    &=\frac{(-1)^k}{k!}g(k)x^k, \tag{3.25}
\end{align*}
where in the first line we have used the analyticity of $g(-z)x^{-z}$ combined with the fact that the gamma function has simple poles at the non-positive integers. The theorem now follows from the Mellin inversion theorem, just as before.
\end{proof}

\begin{center}
    {\large \textsc{\textbf{\S IV. Integral representation of meromorphic function}}}
\end{center}
The aim of this section is to study independently the integral representation of meromorphic functions with simple poles at the non-positive integers that we obtained in \eqref{eq:meromorphic-rep}:
\begin{align}\label{eq:mero-rep}
    h(s)=\int_0^\infty x^{s-1}\sum_{k=0}^\infty\mymathop{Res}_{z=-k}(h(z))x^k\ dx. \tag{4.1}
\end{align}
In particular, we will show that $h(s)$ is logarithmically convex and the function $h_m(x)=h(x+m)/h(m)$ is supermultiplicative. 

Let $f,g,w:I\subseteq \mathbb{R}\to\mathbb{R}$ such that $w(x)\geq 0, \forall x\in I$ and $w, wfg, wf, wg$ are integrable on $I$. We say $f$ and $g$ are \textit{synchronous} if $(f(x)-f(y))(g(x)-g(y))\geq 0, \forall x,y\in I$ and \textit{asynchronous} if $(f(x)-f(y))(g(x)-g(y))\leq 0, \forall x,y\in I$. If $f$ and $g$ are synchronous (asynchronous) on $I$ then \cite{27}
\begin{equation}\label{eq:chebyshev}
    \int_{I}w(x)dx\int_{I}w(x)f(x)g(x)\geq(\leq)\int_{I}w(x)f(x)dx\int_{I}w(x)g(x)dx. \tag{4.2}
\end{equation}

For a given $m\geq 0$, consider $h_m:[0,\infty)\to\mathbb{R}$,
\begin{equation*}
    h_{m}(x)=\frac{h(x+m)}{h(m)}. \tag{4.3}
\end{equation*}
where $h(s)$ is defined as in \eqref{eq:mero-rep}.
\begin{theorem}
   The mapping $h_{m}(.)$ is supermultiplicative on $[0,\infty)$. 
\end{theorem}
\begin{proof}
    Let $f(t)=t^x$ and $g(t)=t^y$ for some $x,y>0$. Then both $f(t)$ and $g(t)$ are non decreasing on $[0,\infty)$ and are therefore synchronous. Suppose that
    \begin{equation*}
        w(t):=t^{m-1}\sum_{k=0}^{\infty}\mymathop{Res}_{z=-k}(h(z))t^k\geq 0 \tag{4.4}
    \end{equation*}
on $[0,\infty)$. Then applying Chebyshev's inequality \eqref{eq:chebyshev} for synchronous mappings produces
\begin{align*}
    \int_{0}^{\infty}t^{m-1}&\sum_{k=0}^{\infty}\mymathop{Res}_{z=-k}(h(z))t^kdt\int_{0}^{\infty}t^{x+y+m-1}\sum_{k=0}^{\infty}\mymathop{Res}_{z=-k}(h(z))t^kdt \\& \geq \int_{0}^{\infty}t^{x+m-1}\sum_{k=0}^{\infty}\mymathop{Res}_{z=-k}(h(z))t^kdt\int_{0}^{\infty}t^{y+m-1}\sum_{k=0}^{\infty}\mymathop{Res}_{z=-k}(h(z))t^kdt, \tag{4.5}
\end{align*}
which is equivalent to
\begin{equation*}
    h(m)h(x+y+m)\geq h(x+m)h(y+m). \tag{4.6}
\end{equation*}
It therefore follows from the definition of $h_m$ that
\begin{equation*}
    h_{m}(x+y)\geq h_{m}(x)h_{m}(y), \tag{4.7}
\end{equation*}
and this completes the proof.
\end{proof}

Let $I\subset\mathbb{R}$ be an interval in $\mathbb{R}$ and assume that $f\in L_{p}(I)$ and $g\in L_{q}(I)$, i.e.,
\begin{equation*}
    \int_{I}|f(s)|^pds, \int_{I}|f(s)|^qds<\infty, \,\,\left(q,p>1\right). \tag{4.8}
\end{equation*}
If $\frac{1}{p}+\frac{1}{q}=1$, then $fg\in L_{1}(I)$ and H\"older's inequality yields
\begin{equation*}
 \left|\int_{I}f(s)g(s)ds\right|\leq \left(\int_{I}|f(s)|^p ds\right)^\frac{1}{p} \left(\int_{I}|g(s)|^q ds\right)^\frac{1}{q}. \tag{4.9}
\end{equation*}
In fact, for some non-negative function $w(s)$ on $I$, the following analogue of H\"older's inequality holds:
\begin{equation}\label{eq:holder-analog}
    \left|\int_{I}f(s)g(s)w(s)ds\right|\leq \left(\int_{I}|f(s)|^p w(s) ds\right)^\frac{1}{p} \left(\int_{I}|g(s)|^q w(s) ds\right)^\frac{1}{q} \tag{4.10}
\end{equation}
provided that the integrals exist and are finite.

\begin{theorem}
    Let $a,b>0$ with $a+b=1$ and $x,y\geq 0$. Then $h(s)$ is logarithmically convex.
\end{theorem}
\begin{proof}
    Choose $f(s)=s^{a(x-1)}$, $g(s)=s^{b(y-1)}$ and 
    \begin{equation*}
         w(s):=\sum_{k=0}^{\infty}\mymathop{Res}_{z=-k}(h(z))s^k\geq 0 \tag{4.11}
    \end{equation*}
     for $s\in(0,\infty)$ in \eqref{eq:holder-analog} with $I=(0,\infty)$ and $p=1/a, q=1/b$. Then
    \begin{align*}
        \int_{0}^{\infty}&s^{a(x-1)}s^{b(y-1)}\sum_{k=0}^{\infty}\mymathop{Res}_{z=-k}(h(z))s^k ds\\&\leq \left(\int_{0}^{\infty}s^{pa(x-1)}\sum_{k=0}^{\infty}\mymathop{Res}_{z=-k}(h(z))s^k ds\right)^\frac{1}{p}\left(\int_{0}^{\infty}s^{qb(y-1)}\sum_{k=0}^{\infty}\mymathop{Res}_{z=-k}(h(z))s^k ds\right)^\frac{1}{q}, \tag{4.12}
    \end{align*}
    from which it follows that
\begin{align*}
        \int_{0}^{\infty}s^{ax+by-1}&\sum_{k=0}^{\infty}\mymathop{Res}_{z=-k}(h(z))s^k ds \\&\leq \left(\int_{0}^{\infty}s^{x-1}\sum_{k=0}^{\infty}\mymathop{Res}_{z=-k}(h(z))s^k ds\right)^a \left(\int_{0}^{\infty}s^{y-1}\sum_{k=0}^{\infty}\mymathop{Res}_{z=-k}(h(z))s^k ds \right)^b . \tag{4.13}
    \end{align*}
Thus, 
\begin{equation*}
    h(ax+by)\leq [h(x)]^a [h(y)]^b. \tag{4.14}
\end{equation*}
    This completes the proof. 
\end{proof}
    When $h(s)=\Gamma(s)$, we recover the above results for the gamma function.

\begin{center}
    \textsc{\textbf{\large \S. V Higher order poles}}
\end{center}
A consequence of the introduction of meromorphic functions with higher order poles is a loss of the power series form for the function undergoing a Mellin transformation. Instead, we see that factors of $\log(x)$ make an appearance in the summand, making the identities that follow rather exotic, and for this reason we forgo a study of the convergence conditions, giving only heuristic results in this section. 

\begin{theorem}\label{thm:higher-order}
Let $g(z)$ be analytic on the half-plane $H(\delta)=\{z\in\mathbb{C}|\textnormal{Re}(z)\ge-\delta\}$ for some $0<\delta<1$, and let $h(z)$ be meromorphic with poles of order $N_k$ at the non-positive integers $k=0,-1,-2,\ldots$. Then under suitable growth conditions,
\begin{align*} \int_0^\infty x^{s-1}\sum_{k=0}^\infty \sum_{n=0}^{N_k-1}\frac{(-1)^{N_k-n-1}}{\Gamma(N_k-n)}c^{(k)}_{n-N_k}\left[\left(\frac{d}{dz}+\log(x)\right)^{N_k-n-1}g(z)\right]_{z=k}x^kdx=h(s)g(-s)\tag{5.1}
\end{align*}
holds for all $0<\textnormal{Re}(s)<\delta$, where $c^{(k)}_{-1},\ldots,c^{(k)}_{-N_k}$ are the coefficients of the principal part of $h$ around $z=-k$.
\end{theorem}
\begin{proof}
    \sloppy The assumption of suitable growth conditions is simply to assure convergence where necessary and to force the arc term in the inverse Mellin transform to vanish in the limit. The remainder of the theorem is proven by examining the Laurent expansion of $h(z)g(-z)x^{-z}$ about the non-positive integers. By assumption, $h$ admits a Laurent series expansion about $z=-k$ of the form
    \begin{equation*}
        h(z)=\sum_{j=-N_k}^\infty c^{(k)}_j(z+k)^j. \tag{5.2}
    \end{equation*}
    Similarly, $g$ admits a Taylor series
    \begin{equation*}
        g(-z)=\sum_{j=0}^\infty(-1)^j\frac{g^{(j)}(k)}{j!}(z+k)^j,\tag{5.3}
    \end{equation*}
    and we may expand $x^{-z}$ as
    \begin{equation*}
        x^{-z}=\sum_{j=0}^\infty (-1)^j\frac{\log^j(x)}{j!}x^k(z+k)^j.\tag{5.4}
    \end{equation*}
    Then the product of the latter two functions is given by the Cauchy product
    \begin{align*}
        g(-z)x^{-z}&=\sum_{j=0}^\infty\sum_{i=0}^j(-1)^j\frac{g^{(i)}(k)\log^{j-i}(x)}{i!(j-i)!}x^k(z+k)^j\\
        &=x^k\sum_{j=0}^\infty\frac{(-1)^j}{j!}\left[\left(\frac{d}{dz}+\log(x)\right)^jg(z)\right]_{z=k}(z+k)^j.\tag{5.5}
    \end{align*}
    It follows that the residue of $h(z)g(-z)x^{-z}$ at $z=-k$ is given by
    \begin{align*}
        \mymathop{Res}_{z=-k}(h(z)g(-z)x^{-z})&=x^k\sum_{n=-N_k}^{-1}c^{(k)}_n\frac{(-1)^{-n-1}}{(-n-1)!}\left[\left(\frac{d}{dz}+\log(x)\right)^{-n-1}g(z)\right]_{z=k}\\
        &=x^k\sum_{n=0}^{N_k-1}c^{(k)}_{n-N_k}\frac{(-1)^{N_k-n-1}}{\Gamma(N_k-n)}\left[\left(\frac{d}{dz}+\log(x)\right)^{N_k-n-1}g(z)\right]_{z=k}.\tag{5.6}
    \end{align*}
    Thus,
    \begin{equation*}
        \frac{1}{2\pi i}\int_C h(z)g(-z)x^{-z}dx=\sum_{k=0}^\infty\sum_{n=0}^{N_k-1}c^{(k)}_{n-N_k}\frac{(-1)^{N_k-n-1}}{\Gamma(N_k-n)}\left[\left(\frac{d}{dz}+\log(x)\right)^{N_k-n-1}g(z)\right]_{z=k}x^k,\tag{5.7}
    \end{equation*}
    from which the Mellin inversion theorem produces
    \begin{align*} \int_0^\infty x^{s-1}\sum_{k=0}^\infty \sum_{n=0}^{N_k-1}\frac{(-1)^{N_k-n-1}}{\Gamma(N_k-n)}c^{(k)}_{n-N_k}\left[\left(\frac{d}{dz}+\log(x)\right)^{N_k-n-1}g(z)\right]_{z=k}x^kdx=h(s)g(-s).\tag{5.8}
\end{align*}
\end{proof}

Observe that if $h$ has only simple poles at the non-positive integers, then its $m$-th derivative $h^{(m)}$ has poles of order $m+1$. That is, the derivative has the effect of increasing the order of poles by one. Moreover, with each derivative, no new coefficients are added to the principal part of the Laurent series expansion, and the coefficient from the simple pole (the residue) gets shifted down an index. Explicitly, if we expand $h$ around $z=a$, so that we have $h(z)=b_{-1}(z-a)^{-1}+O(1)$, it then follows that $h'(z)=-b_{-1}(z-a)^{-2}+O(1)$, and more generally, $h^{(m)}(z)=(-1)^mm!b_{-1}(z-a)^{-m-1}+O(1)$. Thus if $a=-k$, the coefficient $c^{(k)}_{-m-1}$ in Theorem~\ref{thm:higher-order} is given by $(-1)^mm!b_{-1}$ and all other coefficients of the principal part of the expansion around $z=-k$ vanish. We therefore obtain the following corollary.

\begin{corollary}\label{cor:deriv-cor}
Let $g(z)$ be analytic on the half-plane $H(\delta)=\{z\in\mathbb{C}|\textnormal{Re}(z)\ge-\delta\}$ for some $0<\delta<1$, and let $h(z)$ be meromorphic with simple poles the non-positive integers $k=0,-1,-2,\ldots$. Then under suitable growth conditions,
\begin{align*} \int_0^\infty x^{s-1}\sum_{k=0}^\infty \mymathop{Res}_{z=-k}(h(z))\left[\left(\frac{d}{dz}+\log(x)\right)^{m}g(z)\right]_{z=k}x^kdx=h^{(m)}(s)g(-s)\tag{5.9}
\end{align*}
holds for all $0<\textnormal{Re}(s)<\delta$.
\end{corollary}
\begin{proof}
    Since $h$ is meromorphic with simple poles, we may expand it around $z=-k$ so that $h(z)=\mymathop{Res}_{z=-k}(h(z))(z+k)^{-1}+O(1)$, and the $m$-th derivative satisfies $h^{(m)}(z)=(-1)^mm!\mymathop{Res}_{z=-k}(h(z))(z+k)^{-m-1}+O(1)$. Thus, the coeffcient $c^{(k)}_{-m-1}$ of $h^{(m)}$ is given by $(-1)^mm!\mymathop{Res}_{z=-k}(h(z))$ and all other coefficients $c_j$ with $j<0$ vanish. By Theorem~\ref{thm:higher-order}, we have
    \begin{align*} h^{(m)}(s)g(-s)&=\int_0^\infty x^{s-1}\sum_{k=0}^\infty \frac{(-1)^m}{\Gamma(m+1)}c_{-m-1}\left[\left(\frac{d}{dz}+\log(x)\right)^{m}g(z)\right]_{z=k}x^kdx\\
    &=\int_0^\infty x^{s-1}\sum_{k=0}^\infty \mymathop{Res}_{z=-k}(h(z))\left[\left(\frac{d}{dz}+\log(x)\right)^{m}g(z)\right]_{z=k}x^kdx,\tag{5.10}
\end{align*}
and this completes the proof.
\end{proof}

This corollary has several interesting special cases. For example, Corollary~\ref{cor:digamma} can now be generalized to any derivative of the digamma function.

\begin{corollary}
Let $g(z)$ be analytic on the half-plane $H(\delta)=\{z\in\mathbb{C}|\textnormal{Re}(z)\ge-\delta\}$ for some $0<\delta<1$. Then
\begin{align*} \int_0^\infty x^{s-1}\sum_{k=0}^\infty& \bigg[\bigg(\frac{d}{dz}+\log(x)\bigg)^mg(z)\bigg]_{z=k}x^kdx=-\psi^{(m)}(s)g(-s) \tag{5.11}
\end{align*}
holds for all $0<\textnormal{Re}(s)<\delta$.
\end{corollary}
\begin{proof}
     The residue of $\psi(z)$ at $z=-k$ is $1$. The corollary now follows from Corollary~\ref{cor:deriv-cor}.
\end{proof}

Just as for Corollary~\ref{cor:digamma}, this result seems to be more of an interesting artifact of the more general theorem, as finding a suitable $g$ for which the theorem holds seems to be challenging in practice. Another special case of Corollary~\ref{cor:deriv-cor} is given by choosing $h(z)=\frac{\pi}{\sin(\pi z)}$.

\begin{corollary}
Let $g(z)$ be analytic on the half-plane $H(\delta)=\{z\in\mathbb{C}|\textnormal{Re}(z)\ge-\delta\}$ for some $0<\delta<1$. Then
\begin{align*} \int_0^\infty x^{s-1}\sum_{k=0}^\infty(-1)^k \bigg[\bigg(\frac{d}{dz}+\log(x)\bigg)^mg(z)\bigg]_{z=k}x^kdx=\frac{d^m}{ds^m}\left(\frac{\pi}{\sin(\pi s)}\right)g(-s)\tag{5.12}
\end{align*}
holds for all $0<\textnormal{Re}(s)<\delta$.
\end{corollary}
\begin{proof}
     The residue of $\frac{\pi}{\sin(\pi z)}$ at $z=-k$ is $(-1)^k$. The corollary now follows from Corollary~\ref{cor:deriv-cor}.
\end{proof}

This case is a very natural generalization of Ramanujan's theorem in the original form \eqref{eq:rmt}. Recall that the substitution $g(z)\to g(z)/\Gamma(z+1)$ and an application of Euler's reflection formula produces the form \eqref{eq:rmt-gamma}. Another special case given by choosing $h(z)=\Gamma(z)$ can be seen as a natural generalization of this latter form of Ramanujan's master theorem.
\begin{corollary}
Let $g(z)$ be analytic on the half-plane $H(\delta)=\{z\in\mathbb{C}|\textnormal{Re}(z)\ge-\delta\}$ for some $0<\delta<1$. Then
\begin{align*} \int_0^\infty x^{s-1}\sum_{k=0}^\infty\frac{(-1)^k}{k!} \bigg[\bigg(\frac{d}{dz}+\log(x)\bigg)^mg(z)\bigg]_{z=k}x^kdx=\Gamma^{(m)}(s)g(-s)\tag{5.13}
\end{align*}
holds for all $0<\textnormal{Re}(s)<\delta$.
\end{corollary}
\begin{proof}
     The residue of $\Gamma(z)$ at $z=-k$ is $(-1)^k/k!$. The corollary now follows from Corollary~\ref{cor:deriv-cor}.
\end{proof}

Note that by setting $g=1$ the last two corollaries give integral representations of the derivatives of the gamma and cosecant functions. The next corollary takes $h$ to be the square of the gamma function and produces an integral representation involving Euler's constant $\gamma$ and the harmonic numbers $H_k$.
\begin{corollary}\label{cor:gamma-squared}
Let $g(z)$ be analytic on the half-plane $H(\delta)=\{z\in\mathbb{C}|\textnormal{Re}(z)\ge-\delta\}$ for some $0<\delta<1$. Then
\begin{align*} \int_0^\infty \sum_{k=0}^\infty&\frac{g'(k)-(2\gamma -2H_k+\log(x))g(k)}{(k!)^2}x^{k+s-1}dx=(\Gamma(s))^2g(-s)  \tag{5.14}
\end{align*}
holds for all $0<\textnormal{Re}(s)<\delta$.
\end{corollary}
\begin{proof}
    Consider the square of the Gamma function. With the aid of a computer, the residue calculation yields
\begin{align*}
    &\mymathop{Res}_{z=-k}((\Gamma(z))^2g(-z)x^{-z})=\frac{(g'(k)-2\gamma g(k)+2H_kg(k)-\log(x)g(k))}{(k!)^2}x^k,\tag{5.15}
\end{align*}
where $\gamma$ is Euler's constant and $H_k$ is the $k$-th harmonic number. Substituting the above expression in Theorem~\ref{thm:higher-order} yields the desired result.
\end{proof}

With $g=1$, we obtain an integral representation for the gamma function given by
\begin{equation*}
    (\Gamma(s))^2=-\int_0^\infty\sum_{k=0}^\infty\frac{2\gamma-2H_k+\log(x)}{(k!)^2}x^{k+s-1}dx, \tag{5.16}
\end{equation*}
and specializing to $s=1/2$, we recover an integral representation of $\pi$ that involves Euler's constant and the harmonic numbers. Explicitly, we have
\begin{equation*}
    \pi=-\int_0^\infty\sum_{k=0}^\infty\frac{2\gamma-2H_k+\log(x)}{(k!)^2}x^{k-1/2}dx. \tag{5.17}
\end{equation*}
The sum in the integrand can be evaluated explicitly in terms of the modified Bessel function of the second kind, producing
\begin{equation*}
    \frac{\pi}{2}=\int_0^\infty\frac{K_0(2\sqrt{x})}{\sqrt{x}}dx, \tag{5.18}
\end{equation*}
which can be verified numerically with Mathematica. We again recover Bernoulli's integral representation of the gamma function by setting $g(z)=\sin(\pi z)\Gamma(z+1)$, in which case $g(k)=0$ and $g'(k)=\pi(-1)^k k!$, and we have
\begin{equation*}
    \Gamma(s)=\int_0^\infty x^{s-1}e^{-x}dx.\tag{5.19}
\end{equation*}

Another option available to us is to replace $\pi/\sin(\pi s)$ with its $m$-th power $\pi^m/\sin^m(\pi s)$. The residue computation becomes much more complicated because the poles at the non-positive integers are no longer simple; they have order $m$. However, by evaluating the first several cases numerically, we are able to formulate a conjecture for this scenario. Define a polynomial sequence $P_m(x)$ recursively by $P_1(x)=1$, $P_2(x)=x$, and
\begin{align*}
    P_m(x):=(x^2+(m-2)^2\pi^2)P_{m-2}(x) \tag{5.20}
\end{align*}
for $m>2$. These polynomials appear in a work of H. Airault on Fourier transform computations related to hyperbolic measures \cite{0}.

\begin{conjecture} Let $g(z)$ be analytic on the half-plane $H(\delta)=\{z\in\mathbb{C}|\textnormal{Re}(z)\ge-\delta\}$ for some $0<\delta<1$  and let $m$ be a positive integer. Then
\begin{align*} \int_0^\infty x^{s-1}\sum_{n=0}^\infty (-1)^{mn}\left[P_m\left(\frac{d}{dz}+\log(x)\right)g(z)\right]_{z=n}x^ndx&=\frac{(-1)^{m-1}(m-1)!\pi^m}{\sin^{m}(\pi s)}g(-s)\tag{5.21}
\end{align*}
holds for all $0<\textnormal{Re}(s)<\delta$ and some suitable growth conditions on $g$.
\end{conjecture}

Let us give several examples using this conjecture. By setting $g=1$, we recover an interesting integral representation for the $m$-th power of the cosecant function given by
\begin{align*} \int_0^\infty x^{s-1}\sum_{n=0}^\infty (-1)^{mn}P_m\left(\log(x)\right)x^ndx&=(-1)^{m-1}(m-1)!\frac{\pi^m}{\sin^{m}(\pi s)}.\tag{5.22}
\end{align*}
Letting $g(z)=1/\Gamma(z+1)$, we instead find that
\begin{align*} \int_0^\infty x^{s-1}\sum_{n=0}^\infty (-1)^{mn}\left[P_m\left(\frac{d}{dz}+\log(x)\right)\frac{1}{\Gamma(z+1)}\right]_{z=n}x^ndx&=\frac{(-1)^{m-1}(m-1)!\pi^m}{\sin^{m}(\pi s)\Gamma(1-s)}.\tag{5.23}
\end{align*}
For the special case of $m=2$ (the $m=1$ case is Ramanujan's master theorem), Mathematica is able to evaluate the summation explicitly, producing
\begin{align*}
    \sum_{n=0}^\infty \left[P_2\left(\frac{d}{dz}+\log(x)\right)\frac{1}{\Gamma(z+1)}\right]_{z=n}x^n=-e^x\Gamma(0,x),\tag{5.24}
\end{align*}
where $\Gamma(a,z)$ denotes the incomplete gamma function. Thus, the conjecture induces the integral representation
\begin{align*}
    \int_0^\infty x^{s-1}e^{x}\Gamma(0,x)dx=\frac{\pi^2}{\sin^2(\pi s)\Gamma(1-s)},\tag{5.25}
\end{align*}
and after applying Euler's reflection formula, we have
\begin{align*}
    \int_0^\infty x^{s-1}e^{x}\Gamma(0,x)dx=\frac{\pi}{\sin(\pi s)}\Gamma(s). \tag{5.26}
\end{align*}
For $m>2$, Mathematica is both unable to give a closed-form evaluation for the summation and unable to evaluate the integral for arbitrary $s$.

As a final example, let $g(z)=\frac{1}{z+1}$ and set $m=2$. Then the summation is evaluated numerically as
\begin{align*}
    \sum_{n=0}^\infty \left[P_2\left(\frac{d}{dz}+\log(x)\right)\frac{1}{z+1}\right]_{z=n}x^n=\frac{\log(1-x)\log(x)+\mymathop{Li}_2(x)}{x},\tag{5.27}
\end{align*}
where $\mymathop{Li}_n(x)$ denotes the polylogarithm. Thus, the conjecture produces the evaluation
\begin{align*}
    \int_0^\infty x^{s-2}\left(\log(1-x)\log(x)+\mathrm{Li}_2(x)\right)dx=-\frac{\pi^2}{\sin^2(\pi s)}\frac{1}{1-s}.\tag{5.28}
\end{align*}
Similarly, for m = 3, we obtain
\begin{align*}
    \int_0^\infty x^{s-2}\bigg(\pi^2\log(1+x)&+\log^2(x)\log(1+x)\\
    &+2\log(x)\mathrm{Li}_2(-x)-2\mathrm{Li}_3(-x)\bigg)dx=\frac{2\pi^3}{(1-s)\sin^3(\pi s)}.\tag{5.29}
\end{align*}

\begin{center}
    {\large \textsc{\textbf{\S V. Conclusion }}}
\end{center}
In this paper, we have explored an extension of Ramanujan's master theorem by generalizing Hardy's proof to accommodate a broader class of functions. Our examination began with a review of Ramanujan's original theorem and Hardy's rigorous approach, which utilized Cauchy's residue theory and the Mellin inversion theorem. By modifying Hardy's residue calculations, we extended the theorem to a family of Ramanujan-like master theorems, which offer valuable insights and new results in the analysis of Mellin transforms.

Through this generalization, we have derived novel integral representations for meromorphic functions with simple poles at non-positive integers and demonstrated their nice properties, such as logarithmic convexity. Furthermore, we extended our results to meromorphic functions with poles of arbitrary order, uncovering several exotic integral representations that enrich our understanding of these functions.

Our results not only contribute to the theoretical development of mathematical transforms but also have practical implications across various fields. The Mellin transform's applications in computer algorithms, quantum theory, asymptotic analysis, and other domains underscore the significance of our findings. By providing a broader toolkit for evaluating Mellin transforms, our generalization offers new avenues for research and application, reflecting the enduring influence of Ramanujan's insights and the continued relevance of his mathematical legacy. We hope that these results will inspire further exploration and application, continuing the tradition of building upon the foundational work of great mathematicians like Ramanujan.
\newline

\textbf{Acknowledgement:}  ZPB thanks Christophe Vignat for providing the reference to the work of H. Airault.


\begin{thebibliography}{}
\bibitem{0} Helene Airault. Hyperbolic measures, moments and coefficients. algebra on hyperbolic functions. Journal of Functional Analysis, 255(9):2099–2145, 2008. Special issue dedicated to Paul Malliavin.
\bibitem{1} Tewodros Amdeberhan, Olivier R. Espinosa, Ivan Gonzalez, Marshall Harrison, Victor H. Moll, and Armin Straub. Ramanujan’s master theorem. The Ramanujan Journal, 29(1):103–120, 2012.
\bibitem{2} O. Atale. Analytic expressions for some Mellin transforms with their application to prime counting function and interpolation formulas for the zeta function. Palestine Journal of Mathematics, 12(1), 2023.
\bibitem{3} Bruce Berndt. Ramanujan’s Notebooks II. Springer, New York, NY,, 1985.
\bibitem{4} N. Bleistein and R.A. Handelsman. Asymptotic Expansions of Integrals. Holt,
Rinehart and Winston, 1975.
\bibitem{bradshaw2023} Zachary P. Bradshaw. Explorations in mathematical physics: Special functions in quantum theory and Feynman integrals by the method of brackets. PhD Thesis, Tulane University, 2023.
\bibitem{bradshaw2023a} Z. P. Bradshaw, I. Gonzalez, L. Jiu, V. H. Moll, and C. Vignat. Compatibility of the method of brackets with classical integration rules. Open Mathematics 21(1), 20220581, 2023.
\bibitem{5} Zachary P. Bradshaw and Christophe Vignat. An operational calculus generalization of Ramanujan’s master theorem. Journal of Mathematical Analysis and Applications, 523(2):127029, 2023.
\bibitem{6} Muhammed Aslam Chaudhry and Asghar Qadir. Extension of Hardy’s class for Ramanujan’s interpolation formula and master theorem with applications. Journal of Inequalities and Applications, 2012(1):1–13, 2012.
\bibitem{7} L. Debnath and D. Bhatta. Integral Transforms and Their Applications, Third
Edition. Taylor and Francis, 2014.
\bibitem{8} Hongming Ding, Kenneth Gross, and Donald Richards. Ramanujan’s master theorem for symmetric cones. Pacific Journal of Mathematics, 175(2):447–490, 1996.
\bibitem{27} Dragomir, Sever S, Agarwal, R. P and Barnett, Neil S (1999) Inequalities for
Beta and Gamma Functions Via Some Classical and New Integral Inequalities.
RGMIA research report collection, 2 (3). 
\bibitem{9} Benjamin Epstein. Some Applications of the Mellin Transform in Statistics. The Annals of Mathematical Statistics, 19(3):370 – 379, 1948.
\bibitem{10} Ahmed Fitouhi, Kamel Brahim, and Neji Bettaibi. On some q-versions of the Ramanujan master theorem. The Ramanujan Journal, 50(2):433–458, 2019.
\bibitem{11} A. Liam Fitzpatrick and Jared Kaplan. Unitarity and the holographic s-matrix. Journal of High Energy Physics, 2012(10), October 2012.
\bibitem{12} A. Liam Fitzpatrick, Jared Kaplan, Joao Penedones, Suvrat Raju, and Balt C. van Rees. A natural language for ads/cft correlators. Journal of High Energy Physics, 2011(11), November 2011. 17
\bibitem{13} Philippe Flajolet, Xavier Gourdon, and Philippe Dumas. Mellin transforms and asymptotics: Harmonic sums. Theoretical Computer Science, 144(1):3–58, 1995.
\bibitem{14} Philippe Flajolet and Robert Sedgewick. The Average Case Analysis of Algorithms: Mellin Transform Asymptotics. Research Report RR-2956, INRIA, 1996.
\bibitem{15} Ivan Gonzalez, Karen Kohl, Lin Jiu, and Victor H. Moll. An extension of the
method of brackets. Part 1, 2017.
\bibitem{16} Ivan Gonzalez and Victor H. Moll. Definite integrals by the method of brackets. Advances in Applied Mathematics, 45(1):50–73, 2010.
\bibitem{17} Ivan Gonzalez, Victor H. Moll, and Armin Straub. The method of brackets. part 2: Examples and applications, 2010.
\bibitem{18} Ivan Gonzalez. Method of brackets and Feynman diagram evaluation. Nuclear Physics B - Proceedings Supplements, 205-206:141–146, Aug 2010.
\bibitem{19} Ivan Gonzalez and Ivan Schmidt. Optimized negative dimensional integration method (ndim) and multiloop Feynman diagram calculation. Nuclear Physics B, 769(1):124–173, 2007.
\bibitem{20} D. Griffiths. Introduction to Elementary Particles. Wiley, 2020.
\bibitem{hardy1920} G. H. Hardy. On two theorems of F. Carlson and S. Wigert. Acta Mathematica, 42: 327–339, 1920.
\bibitem{21} G. H. Hardy. Ramanujan Twelve Lectures on Subjects Suggested By His Life and Work. Cambridge University Press, 1940.
\bibitem{22} P. A. Martin. Some applications of the Mellin transform to asymptotics of series. In Inverse Scattering and Potential Problems in Mathematical Physics, Frankfurt, 1995. Peter Lang.
\bibitem{23} I. S. Reed. The Mellin Type of Double Integral. Duke Mathematical Journal, 11(3):565 – 572, 1944.
\bibitem{24} R.J. Sasiela. Electromagnetic Wave Propagation in Turbulence: Evaluation and Application of Mellin Transforms. Springer Series on Wave Phenomena. Springer Berlin Heidelberg, 2012.
\bibitem{25} Matthew D. Schwartz. Quantum Field Theory and the Standard Model. Cambridge University Press, 2013.
18
\bibitem{26} R. Wong. Asymptotic Approximations of Integrals. Classics in Applied Mathematics. Society for Industrial and Applied Mathematics, 2001.




 \end{thebibliography}
\end{document}